\documentclass[10pt]{article}

\usepackage{amssymb,amsmath,amsthm}
\usepackage{graphicx}

\newtheorem{theorem}{Theorem}
\newtheorem{lemma}[theorem]{Lemma}
\newtheorem{proposition}[theorem]{Proposition}

\newtheorem{definition}[theorem]{Definition}

\newcommand\R{\mathbb{R}}
\newcommand\Z{\mathbb{Z}}
\newcommand\E{\mathbb{E}}

\newcommand\Prb{\mathbb{P}}
\newcommand\cP{{\mathcal P}}

\newcommand\Fg[1]{Figure~\ref{f:#1}}
\newcommand\Lm[1]{Lemma~\ref{l:#1}}
\newcommand\Th[1]{Theorem~\ref{t:#1}}

\begin{document}

\title{Percolation in the Secrecy Graph}
\author{Amites Sarkar%
\thanks{Department of Mathematics,
Western Washington University, Bellingham, WA 98225, USA. Email: amites.sarkar@wwu.edu}
\and Martin Haenggi%
\thanks{Department of Electrical Engineering,
University of Notre Dame, Notre Dame, IN 46556, USA. Email: mhaenggi@nd.edu}
}

\maketitle

\begin{abstract}
The secrecy graph is a random geometric graph which is intended to model the connectivity of
wireless networks under secrecy constraints. Directed edges in the graph are present whenever
a node can talk to another node securely in the presence of eavesdroppers, which, in the
model, is determined solely by the locations of the nodes and eavesdroppers. In the case of
infinite networks, a critical parameter is the maximum density of eavesdroppers that can be
accommodated while still guaranteeing an infinite component in the network, i.e., the
{\em percolation threshold}. We focus on the case where the locations of the nodes and
eavesdroppers are given by Poisson point processes, and present bounds for different types
of percolation, including in-, out- and undirected percolation.
\end{abstract}

\section{Introduction}

To assess the impact of secrecy constraints in wireless networks, we have recently introduced
a random geometric graph, the so-called {\em secrecy graph}, that represents the network or
communication graph including only links over which secure communication is possible \cite{Hae08}.

We assume that a transmitter can choose the rate such that it can communicate to any receiver
that is closer than any of the eavesdroppers. This way, the secrecy constraint translates into
a simple geometric constraint for secrecy. Natural topics for investigation include the degree
distributions and the threshold at which infinite components cease to exist. Since the resulting
graph is directed, there are different types of components, including in-, out-, and undirected
components. In each case, the percolation threshold (in terms of the density of eavesdroppers)
is different.

In this paper, we give an overview of the progress made in the last three years on the percolation
thresholds for secrecy graphs, introduce new methods, and present improved bounds for the case
where nodes and eavesdroppers form independent Poisson point processes.

\section{Model}
Our model is as follows. Let $\cP$ and $\cP'$ be independent Poisson processes, of intensities 1 and
$\lambda$ respectively, in $\R^d$. The case $d=2$ provides a good example. We will call the points of
$\cP$ {\em black points} and the points of $\cP'$ {\em red points}. Now define a directed graph, the
{\em directed secrecy graph} $\vec{G}_{\rm sec}$, on vertex set $\cP$, by sending a directed edge from
$x\in\cP$ to $y\in\cP$ if there is no point of $\cP'$ in the open ball $B(x,\|x-y\|)$ centered at $x$
with radius $\|x-y\|$. Note that it makes no difference whether we consider open or closed balls since,
with probability~1, there are no two points of $\cP\cup\cP'$ at the same distance from any point of $\cP$.

The motivation for this construction is that $x\in\cP$ can send a message to $y\in\cP$ without being
overheard by an eavesdropper from $\cP'$. For more details, see~\cite{Hae08}, where the model was
originally defined.

Our main aim in this paper is to study the critical value(s) of $\lambda$ for various types of percolation
in $\vec{G}_{\rm sec}$ in the plane (precise definitions will be given later). We will also make some
comments about the situation in higher dimensions.

Let us remark that the indegree and outdegree distributions in $\vec{G}_{\rm sec}$ have been obtained
in~\cite{PiBaWi} and~\cite{Hae08} respectively. We summarize the results below.

\begin{theorem}
The outdegree distribution in $\vec{G}_{\rm sec}$ is geometric with mean $1/\lambda$, and the indegree
$I$ has moment generating function
\[
\E(e^{tI})=\E(e^{V_d(e^t-1)/\lambda}),
\]
where $V_d$ is the random variable representing the volume of a randomly chosen cell in a Voronoi tessellation
associated with a unit intensity Poisson process in $\R^d$. Equivalently, if $f_d(t)$ is the probability
density function of $V_d$, then
\[
\Prb(I=k)=\frac{1}{k!}\int_0^{\infty}f_d(t)e^{-t/\lambda}(t/\lambda)^k\,dt.
\]
\end{theorem}
\begin{proof}
Fix a vertex $x\in\cP$. Label the points of $\cP\cup\cP'\setminus\{x\}=\{y_1,y_2,\ldots\}$ in order of
increasing distance from $x$. Now $x$ has outdegree $k$ if and only if the $k$ nearest points
$y_1,\ldots,y_k$ to $x$ belong to $\cP$ and $y_{k+1}\in\cP'$. The probability of this is
$\left(\tfrac{1}{1+\lambda}\right)^k\tfrac{\lambda}{1+\lambda}$. Consequently, the outdegree distribution is
geometric with mean $1/\lambda$.

For the indegree distribution, we again fix $x\in\cP$, and temporarily rescale the model so that $\cP$ and
$\cP'$ have intensities $1/\lambda$ and $1$ respectively. This does not affect either degree distribution.
The vertex $x$ has indegree $k$ if and only if there are exactly $k$ points of $\cP$ in the Voronoi cell $C$
defined by $\cP'\cup\{x\}$ containing $x$. If $C$ has volume $V$, then
\[
\Prb(C\cap\cP=k)=\tfrac{1}{k!}e^{-V/\lambda}(V/\lambda)^k.
\]
Consequently,
\[
\Prb(I=k)=\frac{1}{k!}\int_0^{\infty}f_d(s)e^{-s/\lambda}(s/\lambda)^k\,ds
\]
so that
\[
\E(I^n)=\sum_{k=0}^{\infty}\frac{k^n}{k!}\int_0^{\infty}f_d(s)e^{-s/\lambda}(s/\lambda)^k\,ds
\]
and
\begin{align*}
\E(e^{tI})&=\sum_{n=0}^{\infty}\sum_{k=0}^{\infty}\frac{t^nk^n}{n!k!}\int_0^{\infty}f_d(s)e^{-s/\lambda}(s/\lambda)^k\,ds\\
&=\int_0^{\infty}f_d(s)e^{-s/\lambda}\sum_{k=0}^{\infty}\frac{1}{k!}(s/\lambda)^ke^{kt}\,ds\\
&=\int_0^{\infty}f_d(s)e^{-s/\lambda}e^{se^{t}/\lambda}\,ds\\
&=\E(e^{V_d(e^t-1)/\lambda}).
\end{align*}
\end{proof}
Unfortunately, $f_d(t)$ is only known when $d=1$, when $f_1(t)=4te^{-2t}$. Consequently, the indegree
distribution in $\vec{G}_{\rm sec}$ remains unknown for $d\ge 2$. However, its mean is of course $1/\lambda$
in all dimensions.

\section{Percolation}

For a model of an infinite undirected random graph, {\em percolation} is said to occur if an infinite
component occurs with positive probability. (In fact, this probability is almost always 0 or 1 by ergodicity --
see \Th{01}.) Since $\vec{G}_{\rm sec}$ is a directed graph, there are several things we could mean
by ``component", which lead to several definitions of percolation. Following~\cite{BBknnp},
we distinguish five distinct events. First, write $G_{\rm sec}$ for the undirected graph obtained from
$\vec{G}_{\rm sec}$ by removing the orientations of the edges and replacing any resulting double edges
by single edges, and $G'_{\rm sec}$ for the undirected graph obtained from
$\vec{G}_{\rm sec}$ by including only those edges $xy$ for which both $\vec{xy}\in\vec{G}_{\rm sec}$ and
$\vec{yx}\in\vec{G}_{\rm sec}$. We write {\bf U} for the event that $G_{\rm sec}$ has an
infinite component, {\bf O} for the event that $\vec{G}_{\rm sec}$ has an infinite out-component,
{\bf I} for the the event that $\vec{G}_{\rm sec}$ has an infinite in-component, {\bf S} for the
event that $\vec{G}_{\rm sec}$ has an infinite strongly connected subgraph, and {\bf B} for the event
that $G'_{\rm sec}$ has an infinite component. Here, an out (resp. in)-component is a subgraph with a spanning
subtree whose edges are all directed away from (resp. towards) a root vertex, and a strongly connected
subgraph is one where there are directed paths from $x$ to $y$ for all $x$ and $y$ in the subgraph.

These types of percolation are of more than just mathematical interest. For instance, events
${\bf O},{\bf I}$ and ${\bf S}$ are relevant for broadcasting, collecting and sharing data, respectively,
and ${\bf B}$ would be relevant if the transmission protocol needed secure transmission in both directions
at each step.

As noted in~\cite{BBknnp}, we have the following implications:
\begin{equation}
{\bf B}\Rightarrow{\bf S}\Rightarrow({\bf I}{\rm \ and\ }{\bf O}),\ \ \ \ \ ({\bf I}{\rm \ or\ }{\bf O})\Rightarrow{\bf U}.
\end{equation}
Let ${\bf X}$ denote any of ${\bf U},{\bf O},{\bf I},{\bf S}$ or ${\bf B}$, and let $p_{\bf X}(\lambda,d)=\Prb({\bf X})$.
The following theorem is a consequence of the ergodicity of the Poisson process, and the fact that percolation is a
translation invariant event.

\begin{theorem}\label{t:01}
For all $\lambda,d$, and all choices of ${\bf X}$, $p_{\bf X}(\lambda,d)$ is either 0 or 1.\qed
\end{theorem}

Since, for a fixed instance of $\cP$, adding points to $\cP'$ can only remove edges from $\vec{G}_{\rm sec}$,
the probability $p_{\bf X}(\lambda,d)$ is non-increasing in $\lambda$. Define the {\em critical intensity}
$\lambda_{{\bf X},d}$ by the formula
\[
\lambda_{{\bf X},d}=\inf\{\lambda:p_{\bf X}(\lambda,d)=0\}=\sup\{\lambda:p_{\bf X}(\lambda,d)=1\}
\]
and write (just for this paper) $\lambda_{\bf X}=\lambda_{{\bf X},2}$. We reiterate that {\em increasing}
$\lambda$ {\em decreases} the probability of percolation, in our formulation of the model. From (1), we have
\begin{equation}
\lambda_{\bf B}\le\lambda_{\bf S}\le\min\{\lambda_{\bf I},\lambda_{\bf O}\},\ \ \ \ \
\max\{\lambda_{\bf I},\lambda_{\bf O}\}\le\lambda_{\bf U}.
\end{equation}

Our first aim is to provide bounds on $\lambda_{\bf X}$. While doing this, we survey various methods that
have been used for other continuum percolation models. All of these are from~\cite{Gil61},~\cite{Hal85b}
and~\cite{Pen96}, on percolation in the Gilbert disc model, and from~\cite{BBknnp} and~\cite{HaMe}, on
percolation in the $k$-nearest neighbour model.

\subsection{Branching processes (\cite{Gil61},~\cite{HaMe},~\cite{Hal85b},~\cite{Pen96})}

Let $\cP$ be a Poisson process. For fixed $r>0$, the Gilbert disc model is obtained by connecting two points
of $\mathcal{P}$ with an edge if the distance between them is less than $r$, and, for a fixed positive integer $k$,
the $k$-nearest neighbour model is obtained by connecting each point $p\in\cP$ to its $k$ {\sl nearest neighbours}:
those points of $\cP$ which are the closest, in the usual euclidean norm, to $p$.

For both the Gilbert disc model and the $k$-nearest neighbour model (the ``traditional models"), the basic
method is as follows. We start with a vertex $x$ of $\cP$, grow the cluster containing $x$ in ``generations",
and compare the growing cluster to a branching process. For the most natural way of doing this (details below),
the branching process has more points than the cluster, so, in all dimensions, if the branching process dies out,
so will the cluster. We can now use classical results which tell us when certain branching processes die out.
Consequently, in all dimensions, branching processes give lower bounds for thresholds in the traditional models,
i.e., they show that for certain parameters, percolation {\em does not} occur.

In the following, we will describe the method for the Gilbert disc model, although it is almost the same as for
the $k$-nearest neighbour model. Assume that the origin $O$ is a point of $\cP$. First pick the points of $\cP$
within distance $r$ of $O$ -- these are the first generation. The second generation are the points of $\cP$
which are each within distance $r$ of some first generation point, but are not in the first generation themselves
(i.e., they are not within distance $r$ of $O$). The third generation are the points of $\cP$ not belonging to
the first two generations, but which are each within distance $r$ of some second generation point, and so on.
The associated branching process is obtained by setting each offspring size distribution to be ${\rm Po}(\pi r^2)$,
so that we are essentially growing the same cluster containing $O$, but ignoring the fact that the various discs
we have scanned for points actually overlap. In~\cite{Gil61}, Gilbert argues that if $\pi r^2\le 1$, the branching
process dies out with probability 1, so that the critical area for percolation is at least 1. When $\pi r^2>1$,
it is possible to calculate (numerically) the probability that the branching process dies out, so this gives an
upper bound on the probability that $O$ belongs to an infinite component. Gilbert also notes the following
improvement. The discs surrounding a point of $\cP$ and its descendant in $\cP$ always intersect in an area of
at least $\alpha=(\tfrac23\pi-\tfrac{\sqrt 3}{2})r^2$,
so we can compare with a branching process whose offspring size distribution is just ${\rm Po}((\pi-\alpha)r^2)$.
This leads to the improved lower bound of $\tfrac{\pi}{\pi-\alpha}\approx 1.642$, which was further improved to
2.184 by Hall~\cite{Hal85b} using multitype branching processes. In Hall's method, the type of a child is just
the Euclidean distance to its parent: children of higher types are likely to have more descendants. We include
a brief description of Hall's modification later.

This method can be used to give an upper bound of $\lambda_{{\bf O},d}\le 1$ for the secrecy graph model.
In fact, for oriented out-percolation, we have the following result.

\begin{proposition}
The probability $\theta_{{\bf O},d}(\lambda)$  that $O$ belongs to an infinite out-component in the secrecy graph
satisfies
\[
\theta_{{\bf O},d}(\lambda)\le \max\{0,1-\lambda\}.
\]
\end{proposition}
\begin{proof}
As in the above proof sketch, we compare the growing cluster, starting at a black point $p\in\cP$, with a branching
process. The number of children in the first generation has distribution given by a geometric random variable with
mean $1/\lambda$. After the $n^{\rm th}$ generation has been completed, we order the points of the $n^{\rm th}$
generation in order of distance from $p$, and begin growing a ball around each point in turn (according to the order).
For each black point $x$, there are two possibilities. First, the ball corresponding to $x$ might encounter a red
point which has already been encountered. If not, the ball will certainly outgrow the region $R$ already scanned
(by points in previous generations, or the current generation). In this case, the number of black points {\em
outside the region} $R$ that we encounter before the first red point (which stops the ball) will again have a
geometric distribution with mean $1/\lambda$. Consequently, the number of children of a black point is always
stochastically dominated by a geometric random variable with mean $1/\lambda$, and generating function
$f(x)=\tfrac{\lambda}{1+\lambda-x}$. A branching process whose offspring size distribution is
given by this geometric random variable has extinction probability 1 if $\lambda\ge 1$, and extinction
probability $\lambda$ if $\lambda\le 1$. (When $\lambda<1$, the extinction probability is given by the
smallest root of $x=f(x)$.) Consequently, the cluster stops growing with probability at least $\lambda$,
and so $\theta_{{\bf O},d}(\lambda)\le 1-\lambda$.
\end{proof}

In higher dimensions, the cluster is approximated better and better by the appropriate branching process,
at least for the Gilbert and $k$-nearest neighbour models. This is because the distances from a point $p\in\cP$
to its two nearest neighbours in $\cP$ converge in distribution to a (common) deterministic limit, and because the
overlap between the balls centered at a parent and at its child gets smaller and smaller, as $d\to\infty$. There is
a slight complication in that the error (between the model and a branching process) is only asymptotically
negligible over finitely many generations. Therefore, in both~\cite{HaMe} and~\cite{Pen96}, oriented lattice
percolation is brought in to establish asymptotic thresholds for percolation. The results are that in sufficiently
high dimension, $k=2$ gives percolation for the $k$-nearest neighbour model, and that the critical volume in
the Gilbert model tends to 1 as $d\to\infty$.

For the secrecy graph, we have

\begin{theorem}\label{t:pain}
If $\lambda\ge 1$, then, for all $d$, $\theta_{{\bf O},d}(\lambda)=0$. If $\lambda<1$, then
$\theta_{{\bf O},d}(\lambda)\to1-\lambda$ as $d\to\infty$.
\end{theorem}

\noindent The first part of the theorem follows from the above proposition, so we assume from now on that $\lambda<1$.

We will prove this theorem in a series of steps, and we will utilize six different branching random walks. The first
is the process $({\bf X^d_n})$. We define ${\bf X^d_0}$ to be the single point at the origin in $\R^d$, which we
will suppose belongs to $\cP$. ${\bf X^d_1}$ is the set of points of $\cP$ that are closer to ${\bf X^d_0}$ than any point
of $\cP'$, ordered according to modulus. Thus the points in ${\bf X^d_1}$ are the out-neighbours of ${\bf X^d_0}$ in
$\vec{G}_{\rm sec}$. We generate the set ${\bf X^d_2}$ by examining the points of ${\bf X^d_1}$ in order, and growing a ball
around each one, capturing black points until the first red point is encountered. We call this {\em scanning} around the
points of ${\bf X^d_1}$. After we have  scanned around each point of
${\bf X^d_1}$, the newly-captured black points (i.e., those not in ${\bf X^d_0}\cup{\bf X^d_1}$) form ${\bf X^d_2}$.
Thus ${\bf X^d_2}$ is the set of out-neighbours of the points of ${\bf X^d_1}$ in $\vec{G}_{\rm sec}$ that are not
out-neighbours of ${\bf X^d_0}$. This time, we order the points of ${\bf X^d_2}$ according to the order in which they
were captured, i.e., they inherit the order of their parents in ${\bf X^d_1}$, and, within sibling groups, they are ordered
by distance to the parent. The set ${\bf X^d_3}$, of not-already encountered out-neighbours of ${\bf X^d_2}$, is generated
in the same way, and the same ordering is imposed upon its members. Continuing in this manner we obtain
$({\bf X^d_n})$. Of course, it is entirely possible that this process terminates after a finite number of steps.

As we have already remarked, as $d\to\infty$, this process more and more resembles the following one. We set ${\bf Y^d_0}$ to
be the single point at the origin in $\R^d$, as before. The set ${\bf Y^d_1}$ is the set of out-neighbours of ${\bf Y^d_0}$ in
$\vec{G}_{\rm sec}$, again as before. However, to generate ${\bf Y^d_2}$, we use a different procedure. Examining the points of
${\bf Y^d_1}$ in order of modulus, for each point, we generate entirely fresh copies of $\cP$ and $\cP'$, and for each point
$y\in{\bf Y^d_1}$, the children of $y$ in ${\bf Y^d_2}$ are the out-neighbours of $y$ in this new copy of $\vec{G}_{\rm sec}$,
once again ordered by distance to the parent. We continue in this manner to obtain $({\bf Y^d_n})$: each time we
scan around a new point, we use a fresh copy of $\cP$ and $\cP'$, and the ordering on the points within each generation
is as before. This process might also terminate after a finite number of steps.

This process can be coupled with the previous one: to get an instance of a subtree of $\vec{G}_{\rm sec}$ from an
instance of $({\bf Y^d_n})$, we simply throw away some of the black points, along with their descendants.
There are two types
of black point which need to be discarded. Firstly, any black point among the process $({\bf Y^d_n})$ which was
born inside a previously scanned region must be excluded from $({\bf X^d_n})$. Secondly, while scanning about a
point $y\in({\bf Y^d_n})$, we stop when we hit the first red point of the new instance of $\cP'$ we are
using. However, we might encounter an old red point, from the original instance of $\cP'$, first. For the sake of
generating $({\bf X^d_n})$, this is where the scanning around $y$ must stop. Hence we must discard all black points captured
after this old red point was encountered. Owing to the existence of fresh red points in already-scanned regions, we might
actually never obtain some points of the original process $({\bf X^d_n})$, but the new set of points will certainly
be a subset (if not a subtree) of $({\bf X^d_n})$.

One thing is clear, however: if, in, say, the first $k$ generations of $({\bf Y^d_n})$, no point (either black or red)
is born inside a previously scanned region, and if no previously encountered red points are encountered during the scanning,
then the processes $({\bf X^d_n})$ and $({\bf Y^d_n})$ will coincide for the first $k$ generations. We will in fact show
that, for fixed $k$, the probability of this tends to 1 as $d\to\infty$. First, however, let us remark that the distribution
of generation sizes in the process $({\bf Y^d_n})$ is known completely. For this, the spatial locations of the points of
$({\bf Y^d_n})$ are irrelevant: all that matters is that the individuals in $({\bf Y^d_n})$ form a branching process,
whose offspring distribution is geometric with mean $\mu=1/\lambda>1$. Consequently (see, for instance~\cite{Wil91}),
\begin{equation}\label{geom}
\Prb(|{\bf Y^d_n}|=j)=
\begin{cases}
\frac{\mu^n-1}{\mu^{n+1}-1}&\text{if $j=0$}\\
\frac{\mu^n(\mu-1)^2}{(\mu^{n+1}-1)^2}\left(\frac{\mu^{n+1}-\mu}{\mu^{n+1}-1}\right)^{j-1}\sim\frac{\mu^n(\mu-1)^2}{(\mu^{n+1}-1)^2}
\exp\left(-\frac{j(1-\lambda)}{\mu^n}\right)&\text{if $j\ge 1$}.\\
\end{cases}
\end{equation}
(Here, the asymptotics are as $n\to\infty$, with $\mu$ and $j$ fixed.)
The expected size of the $n^{\rm th}$ generation is $\mu^n$, and its mass function is geometric, except for the first term.
Moreover, the extinction probability is $\lambda=1/\mu$, corresponding to the percolation probability $1-\lambda$. The idea
of the rest of the argument is that we can essentially let $k\to\infty$ in the preceding discussion, even though, for any
fixed $d$, the processes $({\bf X^d_n})$ and $({\bf Y^d_n})$ will eventually differ with probability 1.

To compare the processes $({\bf X^d_n})$ and $({\bf Y^d_n})$ over the first $k$ generations, we will use the following
well-known lemmas. To simplify their statements, we will, following~\cite{HaMe} and~\cite{Pen96}, scale the processes
$\cP$ and $\cP'$ so that they have intensities $1/\alpha_d$ and $\lambda/\alpha_d$ respectively, where
$\alpha_d=\pi^{d/2}/\Gamma(1+d/2)$ is the volume of a unit $d$ dimensional ball. This doesn't affect the graph
$\vec{G}_{\rm sec}$.

\begin{lemma}\label{l:dist}
Let $d_i$, for $1\le i\le t$, be the distance of the $i^{\rm th}$ nearest point of $\cP$ to the origin in $\R^d$.
Then, as $d\to\infty$, $d_i\to1$ in probability.\qed
\end{lemma}

\begin{lemma}\label{l:intersect}
Let $B_1$ and $B_2$ be balls in $\R^d$ of radii $r_1,r_2\in(0.9,1.1)$. Suppose that the centers of the $B_i$ are
at least 0.9 units apart. Then, as $d\to\infty$, the proportion of the volume of $B_1$ which lies inside $B_2$
tends to zero.\qed
\end{lemma}

\noindent We apply these lemmas to establish the following fact, which will be central to all that follows.

\begin{lemma}\label{l:xapproxy}
Fix $\epsilon>0$ and $k\ge 1$. If $d\ge d_0(\epsilon,k)$, then the probability that $({\bf X^d_n})$ and $({\bf Y^d_n})$
differ in the first $k$ generations is less than $\epsilon$.\qed
\end{lemma}
\begin{proof}
Fix $y\in{\bf Y^d_n}$. Firstly, by \Lm{dist}, all the
children $y_1,\ldots,y_t$ of $y$ lie at distance approximately 1 from $y$, as $d\to\infty$. Secondly, by \Lm{intersect},
the $y_i$ are at distance more than 1 from each other, and from the nearest red point $z$ to $y$ (which is also at
distance about 1 from $y$). Write $B=B(y,||z-y||)$, so that $B$ is the ball generated about $y$ while scanning for
children. Now, while scanning around the children $y_i$ of $y$, we generate certain balls $B_i$ of radius approximately 1,
centered at the $y_i$, which are stopped by red points $z_i$. The balls $B_i$ will intersect each other, and naturally
they will all intersect $B$. However, again by \Lm{intersect}, the volumes of all these intersections will be negligible
compared to the volumes of the balls themselves. Consequently, the $y_i$ are very likely to have disjoint sets of children,
all born outside $B$, and each of the balls $B_i$ will be stopped by a different point $z_i\not=z$, which will also lie
outside $B$. Now, since the offspring distribution of ${\bf Y^d_n}$ is independent of $d$, the probability of having more
than $N$ points in the first $k$ generations of ${\bf Y^d_n}$ can be made less than $\epsilon/2$ by taking $N$ sufficiently
large (depending on $k$ and $\epsilon$ but not on $d$). For fixed $\epsilon$ and $k$, we choose such an $N$, and repeat
the above argument for $k$ generations. In this process, with probability at least $1-\epsilon/2$, there will be at most
$N$ opportunities for red or black points to be born within the ``forbidden" intersections, and at most $N$ opportunities
to encounter previously discovered red points while scanning. In this case, conditioning on the offspring sizes (but not
locations) in the first $k$ generations of ${\bf Y^d_n}$, the probability of each of these events can be made less than
$\epsilon/4N$ by taking $d$ sufficiently large.
\end{proof}

Incidentally, the edge between a child and its parent will be almost orthogonal to each of the edges joining the same
child to its own children, so the points of ${\bf Y^d_k}$ will all lie at about distance about $\sqrt{k}$ from ${\bf Y^d_0}$,
when $d$ is large.

The next step is to project the points of $({\bf Y^d_n})$ onto $\R^2$ using the map $L:\R^d\to\R^2$ defined by
\[
L(x_1,\ldots,x_d)=\sqrt{d}(x_1,x_2).
\]
The reason for the factor $\sqrt{d}$ is the following lemma, taken from~\cite{HaMe}.

\begin{lemma}\label{l:proj}
Suppose ${\bf Y}$ is uniformly distributed on the surface of the ball of radius 1 in $\R^d$. Then, as $d\to\infty$,
the random variable
${\bf Z}=L({\bf Y})$ converges in distribution to the bivariate normal distribution $N(0,I)$ with mean zero and covariance
matrix equal to the $2\times 2$ identity matrix $I$.
\end{lemma}

\noindent {\bf Remark.} Indeed, the density function of ${\bf Z}$ converges pointwise to
\[
f(z_1,z_2)=\frac{1}{2\pi}\exp\left(-\frac{z_1^2+z_2^2}{2}\right)
\]
as $d\to\infty$.

\begin{proof}
The proof of an almost identical statement appears in~\cite{Pen96}, and the result is well-known, but we sketch the proof
nonetheless. If $X_1,X_2,\ldots,X_d$ are independent $N(0,1)$ random variables,
then the $d$-dimensional random vector ${\bf X}=(X_1,X_2,\ldots,X_d)\in\R^d$ has density function
\[
f_d(x_1,x_2,\ldots,x_d)=\frac{1}{(2\pi)^{d/2}}\exp\left(-\frac{x_1^2+x_2^2+\cdots+x_d^2}{2}\right),
\]
which is radially symmetric. Moreover, using Chebyshev's inequality, we see that $\tfrac{1}{d}|{\bf X}|^2=\tfrac{1}{d}(X_1^2+\cdots+X^2_d)$
converges in probability to $1$ as $d\to\infty$, and so $\frac{1}{\sqrt{d}}{\bf X}$ converges in distribution to ${\bf Y}$.
Consequently, the distribution of the first two coordinates of $\sqrt{d}{\bf Y}$ converges (in distribution) to that of $(X_1,X_2)$,
as stated in the lemma.
\end{proof}

Write $({\bf \tilde{Y}_n})$ for the result of projecting the process $({\bf Y^d_n})$ from $\R^d$ to $\R^2$ using the map $L$,
and write $({\bf \tilde{Y}^{\infty}_n})$ for the process in which the offspring size distribution agrees with that of $({\bf \tilde{Y}_n})$ and
$({\bf Y^d_n})$ (i.e., is geometric with mean $1/\lambda$), but where the offsets of each child are independent $N(0,I)$
random variables. The preceding lemma shows that the processes $({\bf \tilde{Y}_n})$ and $({\bf \tilde{Y}^{\infty}_n})$ resemble each other
more and more as $d\to\infty$. Consequently, we will study the process $({\bf \tilde{Y}^{\infty}_n})$ first, and draw conclusions
about the other processes later.

Rather than consider the entire process $({\bf \tilde{Y}^{\infty}_n)}$, we will use a ``truncated" version, and compare with
oriented site percolation on the lattice $\Lambda=\{(i,j)\in\Z^2:i\ge 0,|j|\le i,i+j\in 2\Z\}$, with oriented edges from
$(i,j)$ to $(i+1,j\pm 1)$. Each site $(i,j)$ of $\Lambda$ will correspond to a square
\[
S_{i,j}=[M(i-1/2),M(i+1/2)]\times[M(j-1/2),M(j+1/2)]
\]
in $\R^2$, where $M$ is a large integer which we will choose later. Since the oriented percolation probability is
left-continuous at 1 (see~\cite{Dur}, for instance), we may choose $\delta>0$ such that, for oriented site
percolation on $\Lambda$ with parameter $p\ge 1-3\delta$, the oriented percolation probability (of the event that
there is an infinite directed path starting from the origin) is greater than $1-\epsilon/2$.

A site $(i,j)$ in the oriented percolation process will be deemed ${\bf \tilde{Y}^{\infty}-}$open if we can proceed to both $(i+1,j-1)$ and
$(i+1,j+1)$ from it. However, ``proceed" will mean different things in the cases $(i,j)=(0,0)$ and $(i,j)\not=(0,0)$. Assume, as before,
that the point ${\bf \tilde{Y}^{\infty}_0}$ lies at the origin. The site $(0,0)$ will be ${\bf \tilde{Y}^{\infty}-}$open if and only if
${\bf \tilde{Y}^{\infty}_0}$ has at least $m$ descendants in generation $k$ within the square $S_{1,-1}$, and at least $m$ descendants,
also in generation $k$, within the square $S_{1,1}$. We will only test a subsequent site $(i,j)$ for ${\bf \tilde{Y}^{\infty}-}$openness if at
least one of $(i-1,j+1)$ or $(i-1,j-1)$ is ${\bf \tilde{Y}^{\infty}-}$open. If at least one of these two sites is ${\bf \tilde{Y}^{\infty}-}$open,
then we know that there are $m$ points $z_1,\ldots,z_m$ of $({\bf \tilde{Y}^{\infty}_n})$ in $S_{i,j}$. (If there are more than $m$ such points,
for definiteness let $z_1,\ldots,z_m$ be the closest ones to the center of the square.) Site $(i,j)$ will be ${\bf \tilde{Y}^{\infty}-}$open
if and only if $z_1,\ldots,z_m$ have at least $m$ descendants in generation $k$ (counted from the $z_i$, not from
${\bf \tilde{Y}^{\infty}_0}$) in $S_{i+1,j-1}$, and at least $m$ descendants in generation $k$ in $S_{i+1,j+1}$. We require
lower bounds on the probabilities of sites being ${\bf \tilde{Y}^{\infty}-}$open, and these are provided by the following lemmas.

\begin{lemma}\label{l:origin}
Fix $\lambda<1,\delta>0$ and $m\ge1$. There exist positive integers $k(\lambda,\delta,m)$ and $M(\lambda,\delta,m)$ such that,
with the above definitions, the probability that $(0,0)$ is ${\bf \tilde{Y}^{\infty}-}$open is at least $1-\lambda-\delta$.
\end{lemma}
\begin{proof}
Since the proof of an almost identical statement appears in~\cite{HaMe} (only the offspring size distribution is different),
we will just sketch the argument. For $\lambda<1$, the branching process is supercritical, and by \eqref{geom} we can find a
generation $k'(\lambda,\delta,m)$ so that the probability that there are, say, $N=N(\lambda,\delta,m)$ members in generation
$k'$ is at least $1-\lambda-\delta/4$. These $N$ individuals $z_1,\ldots,z_N$ will all be at distance about $\sqrt{k'}$ from
the origin, and we can ensure that, with probability $1-\lambda-\delta/2$, at least half of them lie within distance
$M=\lceil2\sqrt{k'}\rceil$ of the origin. If we run the process for another $M^2$ generations, then about $\lambda N$ of the
$z_i$ will not have any descendants in generation $k'+M^2$. However, if we pick a random descendant of each remaining $z_j$
in generation $k'+M^2$, there is a positive probability that it will land in $S_{1,1}$ or $S_{1,-1}$, since this descendant
will lie about distance $M$ from $z_j$, which in turn is likely to lie within distance $M$ from the origin. Consequently,
from the independence, if $N$ is large enough, we will have, with probability at least $1-\lambda-\delta$, at least $m$
descendants of ${\bf \tilde{Y}^{\infty}_0}$ in generation $k=k'+M^2$ in each of $S_{1,1}$ and $S_{1,-1}$.
\end{proof}

\begin{lemma}\label{l:Zij}
Fix $\lambda<1$ and $\delta>0$. Then there exist positive integers
\[
m(\lambda,\delta),k(\lambda,\delta){\rm\ and\ }M(\lambda,\delta)
\]
such that, with the above definitions, the probability that $(0,0)$ is ${\bf \tilde{Y}^{\infty}-}$open is at least $1-\lambda-\delta$,
and the probability that a site $(i,j)$ with $(i,j)\not=(0,0)$ is ${\bf \tilde{Y}^{\infty}-}$open is at least $1-\delta$.
\end{lemma}
\begin{proof}
The first part is just \Lm{origin}. For the second part, we modify the proof of \Lm{origin}.
We reduce $\delta$ if necessary so that $\lambda+\delta<1$, and choose $m$ so that $(\lambda+\delta)^m<\delta$.
Starting at an arbitrary point $z$ of ${\bf \tilde{Y}^{\infty}_0}$ in $S_{i,j}$, rather than the center, we choose $k$ and $M$ so that,
with probability at least $1-\lambda-\delta$, $z$ has at least $m$ descendants in generation $k$ (counted from $z$) of
${\bf \tilde{Y}^{\infty}_0}$, in each of $S_{i+1,j-1}$ and $S_{i+1,j+1}$. Applying this to each of $z_1,\ldots,z_m$, with probability
at least $1-(\lambda+\delta)^m>1-\delta$, some $z_i$ has least $m$ descendants in generation $k$, in each of $S_{i+1,j-1}$ and $S_{i+1,j+1}$.
\end{proof}

Define ${\bf \tilde{Y}-}$openness in the obvious manner, using the process $({\bf \tilde{Y}_n})$ rather than $({\bf \tilde{Y}^{\infty}_n})$.
The following lemma is the analogue of \Lm{Zij} for the process ${\bf \tilde{Y}_n}$.

\begin{lemma}\label{l:Zij2}
Fix $\lambda<1$ and $\delta>0$. Then there exist positive integers
\[
m(\lambda,\delta),k(\lambda,\delta),M(\lambda,\delta){\rm\ and\ }d(\lambda,\delta)
\]
such that, with the above definitions, and for $d\ge d(\lambda,\delta)$, the probability that $(0,0)$ is ${\bf \tilde{Y}-}$open
is at least $1-\lambda-\delta$, and the probability that a site $(i,j)$ with $(i,j)\not=(0,0)$ is ${\bf \tilde{Y}-}$open is at
least $1-\delta$.
\end{lemma}
\begin{proof}
This follows from \Lm{Zij} and \Lm{proj}. First, we use \Lm{Zij} with $\delta/2$ in place of $\delta$ to find suitable values
of $k,m$ and $M$. Then we choose $d$ large enough so that the distributions of the positions of the descendants in generation $k$
of $({\bf \tilde{Y}_n})$ and $({\bf \tilde{Y}^{\infty}_n})$ are sufficiently close so as to change the required probabilities
by at most $\delta/2$.
\end{proof}

If we could draw the same conclusion for the projection of the process $({\bf X^d_n})$, we would be done. However,
the process $({\bf X^d_n})$ is harder to analyze, owing to possible interference between steps. To be specific,
denote by ``step $(i,j)$" the procedure whereby we determine, for the projection of the process $({\bf X^d_n})$,
whether or not the site $(i,j)$ is ${\bf \tilde{X}-}$open (with the obvious definition).
It is possible that, during step $(i,j)$, a black point (or a red point) is born in
a region that was scanned as part of a previous step $(i',j')$. It is also possible that a red point, discovered
in some previous step $(i',j')$, is encountered in step $(i,j)$. We need to show that both of these possibilities
can be neglected, and to do this, we will need to know something about the history of $({\bf X^d_n})$. For this, we
will need to modify $({\bf X^d_n})$ slightly to exclude certain undesirable (but unlikely) events.

For the remainder of the proof, $\delta$ and $\lambda$ will be fixed.
The first step is to show that we may assume that, in the first $k$ generations of step $(i,j)$ of $({\bf X^d_n})$,
the total number of
descendants is bounded by an absolute constant $N$, the total volume scanned is bounded by an absolute constant $V$, and the
distance of the $L-$projection of any point (in these first $k$ generations) from $(Mi,Mj)$ is at most an
absolute constant $R$.
(Here, these ``absolute constants" might depend on the (fixed) $\delta$ and $\lambda$, but they don't depend on $d$.)
Indeed, the probabilities of the failures of these conditions can each be made arbitrarily small by taking $N,V$ and
$R$ suitably large. If any of them fail, we modify the process $({\bf X^d_n})$ to terminate at the first failure, and
deem step $(i,j)$ to be a failure. We denote the modified process by $({\bf {X}^*_n})$, and $({\bf \tilde{X}^*_n})$ will
be the $L-$projection of $({\bf {X}^*_n})$ to $\R^2$.

To summarize, we are changing ${\bf X}^d_n$ by deleting some offspring when certain conditions fail. The result, $({\bf {X}^*_n})$,
might not be a subtree of ${\bf X}^d_n$, since, in constructing $({\bf {X}^*_n})$, we might attach points of ${\bf X}^d_n$ which
were deleted from $({\bf {X}^*_n})$ at an earlier stage. Nonetheless, $({\bf {X}^*_n})$ will still be a subtree of $\vec{G}_{\rm sec}$,
and so an infinite path in $({\bf {X}^*_n})$ still implies out-percolation in $\vec{G}_{\rm sec}$.
The preceding discussion, together with \Lm{xapproxy}, proves the following.

\begin{lemma}\label{l:wapproxz}
Fix $\delta>0$ and $k\ge 1$. If $d\ge d_1(\delta,k)$, then the probability that $({\bf \tilde{X}^*_n})$ and $({\bf \tilde{Y}_n})$
differ in the first $k$ generations is less than $\delta$.\qed
\end{lemma}

We have dealt with two ways in which step $(i,j)$ could fail: the processes $({\bf \tilde{X}^*_n})$ and $({\bf \tilde{Y}_n})$
might differ, or $({\bf \tilde{Y}_n})$ might fail to proceed to both $S_{i+1,j-1}$ and $S_{i+1,j+1}$ for some reason involving
only the $k$ generations corresponding to step $(i,j)$. To these we must add two more: the step might fail
because a black or red point might be born in a previously scanned region (from a step $(i',j')$), or a previously discovered
red point (from a step $(i',j')$) might be encountered. If we can show that, conditioned on the process so far, the probability
of each of these two events can be bounded by $\delta$, we will be done. (We perform the steps in the lexicographic order
$(0,0),(-1,1),(1,1,),(2,-2),\ldots$.) The following lemma does just this.

\begin{lemma}\label{l:far} Let the process $({\bf \tilde{X}^*_n})$ be defined as above. Then, during the step $(i,j)$, the
probability, conditioned on the history of $({\bf \tilde{X}^*_n})$ up to step $(i,j)$, that either a black or a red point is
born in a region scanned in a previous step, or that a red point from a previous step is encountered, is at most $\delta$.
\end{lemma}
\begin{proof}
Consider a previous step $(i',j')$, and suppose that $(Mi',Mj')$ is at distance $x\gg 2R$ from $(Mi,Mj)$. The total volume
scanned in step $(i',j')$ is at most $V$. Some of this scanned volume falls, when projected, into $S_{i,j}$.
However, the projected distance from the center $(Mi,Mj)$ of $S_{i,j}$ to any
point around which scanning has taken place during step $(i',j')$ is at least $x-R$. Consequently, from \Lm{proj}, and
the faster than exponential decay of the normal distribution, if  $x\ge D$ is sufficiently large, at most $\delta'/x^3$ of
the volume scanned in step $(i',j')$ falls, when projected, within distance $M+R$ from the center $(Mi,Mj)$ of $S_{i,j}$.
Summing over all square centers $(Mi',Mj')$ at distance more than $D$ from $(Mi,Mj)$, the total previously scanned volume
from these distant steps (where $x\ge D$) which falls, after projection, within distance $M+R$ from $(Mi,Mj)$ is at most
$\delta'$. Since there are at most $N$ individuals in step $(i,j)$, we can choose $\delta'$ (and hence $D$), so that the
probability that a black or red point from step $(i,j)$ is born in this region of volume $\delta'$ during step $(i,j)$ is
at most $\delta/3$. Similarly, the probability that, while scanning in step $(i,j)$, we hit a previously encountered red
point from a distant step is at most $\delta/3$. (This is because the region we scan in step $(i,j)$ that lies (when
projected) at distance between $x$ and $x+1$ from $(Mi,Mj)$ has volume at most $\delta''/x^2$, and so, by integration,
this random region is unlikely to contain any previously discovered points at projected distance more than $D$ from $(Mi,Mj)$.)
For the steps at distance at most $D$, we can bound the probability of failure of either type by $\delta/3$, because only
boundedly many steps are involved.
\end{proof}

Together, the last three lemmas prove the following one.

\begin{lemma}\label{l:Wij}
Fix $\lambda<1$ and $\delta>0$. Then there exist constants
\[
m(\lambda,\delta),k(\lambda,\delta),M(\lambda,\delta),N(\lambda,\delta),V(\lambda,\delta),
R(\lambda,\delta){\rm\ and\ }d(\lambda,\delta)
\]
such that, with the above definitions, the probability that $(0,0)$ is ${\bf \tilde{X}^*-}$open is at least $1-\lambda-2\delta$,
and the probability that a site $(i,j)$ with $(i,j)\not=(0,0)$ is ${\bf \tilde{X}^*-}$open is at least $1-3\delta$.\qed
\end{lemma}

It only remains to put the pieces together. Given $\epsilon>0$, we choose $\delta<\epsilon/4$ so that, for oriented site
percolation on $\Lambda$ with parameter $p\ge 1-3\delta$, the oriented percolation probability (of the event that
there is an infinite directed path starting from the origin) is greater than $1-\epsilon/2$. From the previous lemma,
we find an infinite directed ${\bf \tilde{X}^*-}$path from the origin, corresponding to an infinite out-component in
$\vec{G}_{\rm sec}$, with probability at least
\[
(1-\lambda-2\delta)(1-\epsilon/2)>(1-\lambda-\epsilon/2)(1-\epsilon/2)>1-\lambda-\epsilon,
\]
as required. This completes the proof of \Th{pain}.

In two dimensions, it should be possible to improve the bound in Proposition 3 using Hall's modification, which,
for the disc model, runs as follows. Each offspring $y$ is indexed by its
distance $t$ to its parent $x$, and its offspring size distribution is bounded in terms of the area of the lune
$B(y,r)\setminus B(x,r)$. In addition, the distribution of the {\em types} of these offspring is also bounded in
terms of the same lune. Consequently, one can compare the growing cluster with an appropriate multitype
branching process (the types are indexed by $t$). For the secrecy graph, there are three parameters one might
wish to keep track of (instead of just one). These are: the radius $r$ of the disc centered at $x$, the distance $t$
of $x$ to its offspring $y$, and the location of the red point $z$ on the boundary $\partial B(x,r)$ of $B(x,r)$.
Nonetheless, one could in principle compute the appropriate conditional probability distribution and this should
result in a slightly improved upper bound.

To summarize, although branching processes are usually employed to show that percolation {\em does not occur}
in these models, they can also be used to show that percolation {\em does occur} for certain fixed values of the
parameters, as $d\to\infty$. For the secrecy graph model, it would be interesting to investigate the case $\lambda=1$,
as $d\to\infty$. Also, the proof of \Th{pain} seems to suggest that the convergence of $\lambda_{{\bf O},d}$ to 1 is
exponential, and it would be interesting to investigate this further.

\subsection{Lattice percolation (\cite{Gil61},~\cite{HaMe},~\cite{Hal85b},~\cite{PiWi},~\cite{PiWi2})}

Two variants of the basic method, applied to the Gilbert model, are described in Gilbert's original paper~\cite{Gil61}.
For both variants, fix a connection radius $r$. First, if we consider the square lattice with bonds of length $r/2$,
and make the state of a bond $e$ open iff there is at least one point of $\cP$ in the square whose {\em diagonal} is $e$,
then bond percolation in the lattice implies percolation in the Gilbert model. Second, if we consider the hexagonal
lattice where the hexagons have side length $r/\sqrt{13}$, and make the state of a hexagon open iff it contains a point
of $\cP$, then face percolation in the hexagonal lattice implies percolation in the Gilbert model. Using the fact that
the critical probabilities for both bond percolation in the square lattice and face percolation in the hexagonal lattice
are equal to $1/2$, one thus obtains upper bounds on the critical area $\pi r^2_c$ of about $17.4$ and $10.9$, respectively.
The latter value was improved to 10.588 by Hall~\cite{Hal85b} using ``rounded hexagons".

H\"aggstr\"om and Meester~\cite{HaMe} used this method to show that, for fixed $d$, percolation occurs in the $k$-nearest
neighbour model for sufficiently large $k$. Pinto and Win~\cite{PiWi} (see~\cite{PiWi2} for more details) applied it
to show that percolation occurs in all versions of the secrecy graph model when $\lambda$ is sufficiently small. For the
latter application, one needs to use {\em dependent percolation}, which means that the bounds are rather weak. In the same
paper, Pinto and Win prove an upper bound on $\lambda_{\bf U}$, also using lattice percolation. Their method is to tile
the plane with regular hexagons, each of side length $\delta$. Divide each hexagon into 6 equilateral triangles in the
obvious way. Set the state of a hexagon to be closed if it contains no black points and at least one red point in each
of its 6 triangles, and open otherwise. If the probability $g(\lambda,\delta)$ of this is at least $1/2$, the critical
probability of face percolation on the hexagonal lattice, then the origin will almost surely be surrounded by arbitrarily
large closed circuits. It is easy to check that an edge of $G_{\rm sec}$ cannot cross a closed circuit, and so percolation
will not occur in $G_{\rm sec}$ if $g(\lambda,\delta)\ge 1/2$. Now
\[
g(\lambda,\delta)=\left(1-e^{-\lambda\sqrt{3}\delta^2/4}\right)^6e^{-3\sqrt{3}\delta^2/2},
\]
and, for fixed $\lambda$, we maximize $g(\lambda,\delta)$ by setting
\[
e^{-\lambda\sqrt{3}\delta^2/4}=\frac{1}{1+\lambda},
\]
so the smallest value of $\lambda$ for which
\[
\left(\frac{\lambda}{1+\lambda}\right)^6\left(\frac{1}{1+\lambda}\right)^{6/\lambda}\ge\frac12
\]
will be an upper bound for $\lambda_{\bf U}$. The last equation can be solved numerically to yield the bound
$\lambda_{\bf U}\le 40.9$.
The method can easily be modified to give bounds for the other $\lambda_{\bf X}$, but we expect that the results
will be rather weak.

In summary, lattice percolation has generally been used to show that percolation {\em does occur} in these models, although
Pinto and Win also used it to show that percolation {\em does not occur} in the secrecy graph if $\lambda$ is
sufficiently large.

\subsection{The rolling ball method (\cite{BBknnp})}

This is a method designed to show that percolation {\em does occur} for certain parameter ranges in various models.
It was applied in~\cite{BBknnp} to prove upper bounds for critical values of $k$ in the $k$-nearest neighbour model.
Unfortunately, when applied to the Gilbert disc model, it only yields an upper bound (on $\pi r_c^2$) of about 12,
worse than the previously best known bound.

The method involves comparison with 1-independent percolation and carries through almost entirely for the
secrecy graph. We will only need to modify some of the equations from~\cite{BBknnp}: however, for completeness,
we include a full account of the method here. First, we state precisely what we mean by a {\em 1-independent
percolation model}.

\begin{definition}
A bond percolation model on $\Z^2$ is said to be 1-independent if, whenever $E_1$ and $E_2$ are sets of edges at graph
distance at least 1 from each other (i.e., if no edge of $E_1$ is incident to any edge of $E_2$), the state of
the edges in $E_1$ is independent of the state of the edges in $E_2$.
\end{definition}

\noindent We will use the following theorem about such models, proved in~\cite{BBW05}.

\begin{theorem}\label{t:onedep}
If every edge in a 1-independent bond percolation model on $\Z^2$ is open with probability at least 0.8639,
then, almost surely, there is an infinite open component. Moreover, if $B$ is a bounded region of the plane,
there is, almost surely, a cycle of open edges surrounding $B$.\qed
\end{theorem}

We will use the first part of the theorem for our lower bounds, and the second part for our upper bounds.

\begin{figure}
\centerline{\includegraphics[width=0.75\columnwidth]{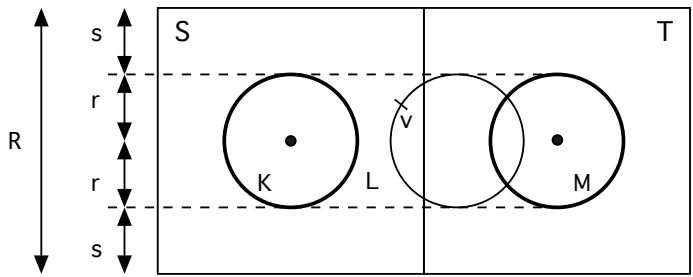}}
\caption{The rolling ball method}
\label{f:1}
\end{figure}

For simplicity, let us first consider the case of ${\bf B}$-percolation. Later, we will indicate the modifications
necessary for the other types.

Consider the rectangular region consisting of two adjacent squares $S,T$ shown in \Fg{1}. Both $S$ and $T$ have side
length $2r+2s$, and $K$ and $M$ have radius $r$ and are placed at the centers of $S$ and $T$ respectively, so that
they each lie at distance $s$ from the boundaries of $S$ and $T$. The parameters $r$ and $s$ will be chosen later.
Also, $T$ may be to the right, left, above or below $S$,
in which case \Fg{1} should be rotated accordingly. We define the {\em basic good event} $E_{{\bf B},S,T}$ to be the
event that every black point $u$ in the central disc $K$ of $S$ is joined to at least one black point in the central
disc $M$ of $T$ by a path in $G'_{\rm sec}$, regardless of the state of the Poisson processes outside $S\cup T$, and
moreover that $K$ contains at least one black point.

Now consider the following percolation model on $\Z^2$. Each vertex $(i,j)\in\Z^2$ corresponds to a square
$[Ri,R(i+1)]\times[Rj,R(j+1)]$ in $\R^2$, where $R=2r+2s$, and an edge is open between adjacent vertices (corresponding
to squares $S$ and $T$) if {\em both} the corresponding basic good events $E_{{\bf B},S,T}$ and $E_{{\bf B},T,S}$ hold.
Note that this is a 1-independent model on $\Z^2$, and that percolation in this model implies percolation in the
original one. Since, by \Th{onedep}, the critical probability for any 1-independent model is
at most 0.8639, if we can show that, for some $r,s,\lambda$,
\[
\Prb(E_{{\bf B},S,T})\ge 0.93195
\]
it will follow that
\[
\Prb(E_{{\bf B},S,T}\cap E_{{\bf B},T,S})\ge 0.8639
\]
by symmetry, and hence we will have shown that $\lambda_{\bf B}\ge\lambda$.

To bound the probability that a basic good event fails, we proceed as follows. Let $K,L$ and $M$ be as in \Fg{1}.
($L$ is the region between the two discs $K$ and $M$.) Define $E'_{{\bf B},S,T}$ to be the event that for every
black point $v\in K\cup L$, there is a black point $u$ such that i) $uv\in E(G'_{\rm sec})$ ii) $\|u-v\|\le s$ and
iii) $u\in D_v$, where $D_v$ is the disc of radius $r$ inside $K\cup L\cup M$ with $v$ on its $K$-side boundary
(the middle disc in \Fg{1}). If we let $F_{S}$ be the event that there is at least one black point in $K$,
then we have (see~\cite{BBknnp} for background)
\[
E'_{{\bf B},S,T}\cap F_S\subset E_{{\bf B},S,T}
\]
and so
\[
E_{{\bf B},S,T}^C\subset (E'_{{\bf B},S,T})^C\cup F_S^C
\]
so that, since $\Prb((E'_{{\bf B},S,T})^C)$ is bounded by the expected number of points $v$ such that i), ii)
or iii) fail,
\[
\Prb(E_{{\bf B},S,T}^C)\le e^{-\pi r^2}+2r(2r+2s)p_{{\bf B},r,s}
\]
where $p_{{\bf B},r,s}$ is the probability that i), ii) or iii) fail for some fixed $v$. Note that this probability
is independent of the position of $v$.

To bound $p_{{\bf B},r,s}$, we consider the probability that the vertex $u$ closest to $v$ inside $D_v$ fails
one of i), ii) or iii) (or does not exist). Suppose some $u\in D_v$ does exist, and write $t=\|u-v\|,A=B(v,t),
B=B(v,t)\cap D_v$ and $C=B(u,t)$. Let $p_{{\bf B}}(u)$ be the probability that $u$ is the closest point to $v$ inside
$D_v$, but that $uv\not\in G'_{\rm sec}$. Then
\begin{equation}\label{b}
p_{\bf B}(u)=(1-e^{-\lambda|A\cup C|})e^{-|B|}
\end{equation}
and also
\[
p_{{\bf B},r,s}\le e^{-|D_v\cap B(v,s)|}+\int_{u\in D_v\cap B(v,s)}p_{\bf B}(u)\,du
\]
so that
\begin{equation}\label{big}\small{
\Prb(E_{{\bf B},S,T}^C)\le e^{-\pi r^2}+2r(2r+2s)\left(e^{-|D_v\cap B(v,s)|}+\int_{u\in D_v\cap B(v,s)}
(1-e^{-\lambda|A\cup C|})e^{-|B|}\,du\right)}
\end{equation}
and the right hand side can be minimized (using a computer) over all $r$ and $s$, with $\lambda$ fixed.
The result for $\lambda=0.0005$ is shown in Table 1, in row ${\bf B}$.

The calculation for the cases ${\bf U}$ and ${\bf O}$ is exactly analogous, using the graphs $G_{\rm sec}$ and
$\vec{G}_{\rm sec}$ respectively. The analogues of~\eqref{b} are
\begin{equation}\label{u}
p_{\bf U}(u)=(1-e^{-\lambda|A|}-e^{-\lambda|C|}+e^{-\lambda|A\cup C|})e^{-|B|}
\end{equation}
and
\begin{equation}\label{o}
p_{\bf O}(u)=(1-e^{-\lambda|A|})e^{-|B|}
\end{equation}
respectively, and the natural analogue of~\eqref{big} applies. The results of the optimization, again obtained using
a computer, are shown in Table 1.

\begin{table}
\[\begin{array}{|c|c|c|c|c|}\hline
{\bf X}&\lambda&r&s&p\\\hline
{\bf U}&0.002&1.659&3.15&0.0669\\
{\bf O}&0.0008&1.658&3.15&0.0677\\
{\bf B}&0.0005&1.657&3.15&0.0680\\
\hline
\end{array}\]
\caption{Upper bounds on $p=\min_{r,s}\Prb(E_{{\bf X},S,T}^C)$. (All values of $p$ rounded up.)}
\end{table}

As proved in~\cite{BBknnp}, the bound for $\lambda_{\bf O}$ in fact applies to $\lambda_{\bf S}$ and $\lambda_{\bf I}$
as well (see ~\cite{BBknnp} for a proof). In conclusion, we have proved the following theorem.

\begin{theorem}
$\lambda_{\bf U}\ge 0.002,\lambda_{\bf O},\lambda_{\bf I},\lambda_{\bf S}\ge 0.0008$ and~$\lambda_{\bf B}\ge~0.0005.$\qed
\end{theorem}

\subsection{High confidence results (\cite{BBknnp})}

This method was used in~\cite{BBknnp} to give both upper and lower bounds for percolation thresholds in the $k$-nearest
neighbour model. It involves computing a certain high dimensional integral using Monte Carlo methods, and so is not fully
rigorous. The approach carries over essentially completely for the secrecy graph.

The lower bound method (corresponding to the upper bound method for the $k$-nearest neighbour model) may be summarized
as follows. Given a trial value of $\lambda$, which we wish to show is a lower bound on one of the percolation thresholds
$\lambda_{\bf U},\lambda_{\bf O}$ or $\lambda_{\bf B}$, we choose trial values of $r$ and $s$. Then we generate a random
instance of $\cP\cup\cP'$ inside $S\cup T$ (see \Fg{1}) and test for the following conditions: i) for more than half of
the black points $v\in K$, there are paths (in $G_{\rm sec},\vec{G}_{\rm sec}$ or $G'_{\rm sec}$ for the cases
${\bf X}={\bf U},{\bf O},{\bf B}$) to more than half the black points in $M$, regardless of the state of $\cP\cup\cP'$
outside $S\cup T$; ii) for more than half of the black points $v\in M$, there are paths to more than half the black points
in $K$, regardless of the state of $\cP\cup\cP'$ outside $S\cup T$. As before, it is clear that this is a 1-independent
model on the bonds joining adjacent squares, and that percolation in this model implies percolation in the original one.
Consequently, if these conditions hold with probability at least 0.8639, then percolation occurs. The condition that the
path should be independent of the process outside $S\cup T$ is simply obtained by ignoring any edges of
$uv\in E(\vec{G}_{\rm sec}(S\cup T))$ where $\|u-v\|>{\rm dist}(u,\partial(S\cup T))$, since only edges $uv$ with
$\|u-v\|\le{\rm dist}(u,\partial(S\cup T))$ are guaranteed to exist in $\vec{G}_{\rm sec}$.

The probability that conditions i) and ii) are satisfied can be expressed as a complicated multiple integral, whose value
we would like to be greater than 0.8639, for some $r$ and $s$. This is the integral we estimate using Monte Carlo methods.
Using a computer program we generated many instances, and counted the proportion of times these conditions held.
From these we calculated the confidence level, i.e., the probability $p$ that these results (or better) could be
obtained, if the true value of the integral was less than 0.8639. In all cases $p$ was less than $10^{-25}$: the detailed
results appear in Table 2. It was shown in~\cite{BBknnp} that the method for the ${\bf X}={\bf O}$ case actually applies to
the cases ${\bf X}={\bf S}$ and ${\bf X}={\bf I}$ as well, so that the results obtained are as follows.

\begin{theorem}\label{t:lb}
{\small With high confidence, $\lambda_{\bf B}\ge 0.09, \lambda_{\bf O},\lambda_{\bf I},\lambda_{\bf S}\ge 0.11,
\lambda_{\bf U}\ge 0.20$.}\qed
\end{theorem}

The upper bound method (corresponding to the lower bound method for the $k$-nearest neighbour model) is as follows.
For suitable $r$ and $s$, we generate instances of $\cP$ and $\cP'$ in $S\cup T$, and check whether, regardless of the
state of the processes outside $S\cup T$, there is no path (in $G_{\rm sec},\vec{G}_{\rm sec}$ or $G'_{\rm sec}$ for
the cases ${\bf X}={\bf U},{\bf O},{\bf B}$) from outside $S\cup T$ that crosses the line segment joining the center
of $S$ to the center of $T$ (see \Fg{2}). We define a 1-independent percolation model on $\Z^2$ by declaring an edge
open if this condition holds for the corresponding rectangle $S\cup T$. If an edge is open with probability at least
0.8639, then, from \Th{onedep}, there are open cycles surrounding any bounded region of the plane. Consequently, if
there was an infinite ${\bf X}$-component starting in some such bounded region, it would have to cross an open cycle,
and in particular cross the central line segment in one of the rectangles $S\cup T$ corresponding to an open edge in
this cycle. This contradicts the condition for that edge to be open, and so percolation cannot occur if the edges are
open with probability at least 0.8639.

It remains to specify how we tested whether an edge of a path (in $G_{\rm sec},\vec{G}_{\rm sec}$ or $G'_{\rm sec}$ for
the cases ${\bf X}={\bf U},{\bf O},{\bf B}$) could come from outside $S\cup T$ to some $v\in S\cup T$. In these cases,
we must find possible neighbours within $S\cup T$ of every possible point outside $S\cup T$. To do this, we used the
following procedure.

We will define a region $R'$, determined by the positions of the red points, so that any black point that is joined
to a point outside of $S\cup T$ must lie in $R'$.
First we define the subset $R\subset R'$ to be the union of various
half-discs $R_i$, described as follows. A point $x$ moving along the boundary of $S\cup T$ has, at almost every position,
exactly one nearest neighbour in $\cP'\cap(S\cup T)$. At some places, there will be a tie for the nearest neighbour of $x$,
so that $\|x-a\|=\|x-b\|$ for some points $a,b\in\cP'$. Draw the disc through $a$ and $b$ and centered at $x$,
and let $R_i$ be the intersection of this disc with $S\cup T$. $R$ is just the union of all such regions $R_i$,
and $R'$ is the union of $R$ together with the regions at the corners of $S\cup T$ which lie outside the $R_i$ (see \Fg{3}).

\begin{figure}
\centerline{\includegraphics[width=0.4\columnwidth]{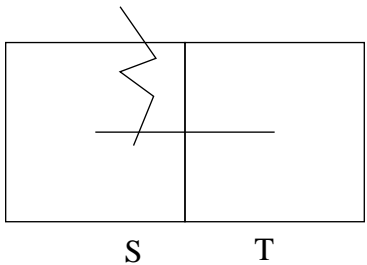}}
\caption{Forbidden path for upper bound method}
\label{f:2}
\end{figure}

\begin{figure}
\centerline{\includegraphics[width=0.8\columnwidth]{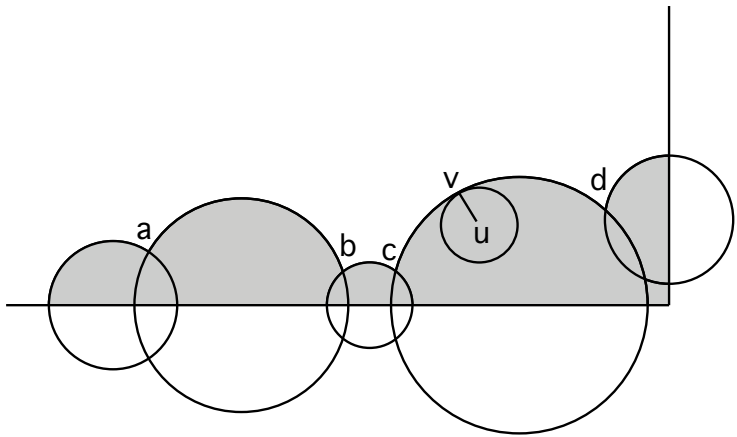}}
\caption{The construction of $R'$ (shaded)}
\label{f:3}
\end{figure}

To check that this method works, suppose that there is an edge $\vec{xy}\in E(\vec{G}_{\rm sec})$, where
$x\not\in S\cup T$ and $y\in (S\cup T)\setminus R$. Let $u$ be the point on $\partial(S\cup T)$ on the line joining
$x$ and $y$. Then, if $B(u,\|u-y\|)$ contains a red point $a$, so does $B(x,\|x-y\|)$, since
\[
\|x-a\|\le \|x-u\|+\|u-a\|< \|x-u\|+\|u-y\|=\|x-y\|
\]
so that it is enough to assume that $x=u$. Moreover, let $v$ be the point on $\partial R$ on the line joining
$u$ and $y$. If $B(u,\|u-v\|)$ contains a red point $b$, so does $B(u,\|u-y\|)$, since $B(u,\|u-y\|)$ contains
$B(u,\|u-v\|)$. Hence we may also assume that $y=v$. Now, with $u$ fixed, we may assume that $v$ is the closest
point of $\partial R$ to $u$, which we may also assume does not coincide with the location of a red point. Draw
the disc $B(u,\|u-v\|)$. By construction, this disc is tangent to one of the half-discs $R_i$, centered at $z$,
say, and has a strictly smaller radius than that of $R_i$, with probability 1. Therefore, its center, $u$, lies
in the interior of the line segment joining $z$ to $v$. Consequently, $u\in S\cup T$, which is a contradiction.
\Fg{3} shows that the three conditions i) $v$ is the closest point of $\partial R$ to $u$ ii)
$\|u-v\| < {\rm min}(\|u-a\|,\|u-b\|)$ and iii) $u\in\partial (S\cup T)$ are incompatible, by illustrating a
typical situation where i) and ii) are satisfied.

In the simulations, points were placed randomly in $S\cup T$, all black points in $R'$ were assumed to be joined
to points outside of $S\cup T$, and edges in $\vec{G}_{\rm sec}$ were determined assuming that there were no red
points outside $S\cup T$. The results of these simulations are also shown in Table 2, and so we have the following
result.

\begin{theorem}\label{t:ub}
{\small With high confidence, $\lambda_{\bf B}\le 0.13, \lambda_{\bf O},\lambda_{\bf I},\lambda_{\bf S}\le 0.17,
\lambda_{\bf U}\le 0.27$.}\qed
\end{theorem}

\begin{table}
\[\begin{array}{|c|c|c|c|c|c|c|c|}\hline
{\bf X}&{\rm bound}&\lambda&r&s&{\rm successes}&{\rm trials}&{\rm confidence}\\\hline
{\bf U}&{\rm lower}&0.20&90&10&1480&1500&10^{-66}\\
{\bf O}&{\rm lower}&0.11&60&0&963&1000&10^{-25}\\
{\bf B}&{\rm lower}&0.09&80&0&2159&2250&10^{-51}\\\hline
{\bf U}&{\rm upper}&0.27&110&0&4296&4600&10^{-51}\\
{\bf O}&{\rm upper}&0.17&110&0&3689&4000&10^{-25}\\
{\bf B}&{\rm upper}&0.13&125&0&6226&6750&10^{-45}\\
\hline
\end{array}\]
\caption{Results of Monte-Carlo simulations. (All confidences rounded up.)}
\end{table}

\section{Uniqueness of the infinite cluster}

Uniqueness of the infinite cluster above the percolation threshold was proved by Harris~\cite{Ha} for bond percolation in $\Z^2$,
by Aizenman, Kesten and Newman~\cite{AKN} for connected, transitive and amenable graphs, by Meester and Roy~\cite{MR} for the Gilbert
model, and by H\"aggstr\"om and Meester~\cite{HaMe} for the $k$-nearest neighbour model. The last two results were obtained by modifying
a very short and elegant argument of Burton and Keane~\cite{BK}, which was originally applied to give a second proof of the
Aizenman--Kesten--Newman theorem. The Burton--Keane argument goes through for the secrecy graph, with a considerably simpler proof
than in~\cite{HaMe}. Before presenting it, we make a few preliminary remarks.

There are three main ingredients in proving the uniqueness of the infinite cluster. One is {\em ergodicity}, which allows us to show that
the number of infinite components is almost surely constant (this constant might be $\infty$). The second is the {\em local modifier},
which works as follows. Suppose we know that some event $E$ occurs with positive probability. Suppose also that, by removing a finite
number of points from any instance of $\cP\cup\cP'$ in which $E$ occurs, we get a configuration in which some other event $F$ always
occurs. Then also $\Prb(F)>0$. This is proved using coupling. The third ingredient is the {\em trifurcation argument}, which, roughly
speaking, shows that the probability of having some infinite component with three distinct ``branches" going off to infinity is zero.
Since the ergodicity and trifurcation arguments are fairly standard (see~\cite{BoRi,HaJo,MRbook} for instance), we will simply state
their implications, without proof, and concentrate on the local modifier.

To keep things simple, we will focus on the case ${\bf X}={\bf B}$. In other words, we will work with the graph $G'_{{\rm sec}}$
of bidirectional edges.
From now on, we will call this graph $G$. Versions of the result, with almost identical proofs, exist for the cases ${\bf X}={\bf U}$
and ${\bf X}={\bf S}$; when ${\bf X}={\bf I}$ or ${\bf X}={\bf O}$, things are more complicated, since two maximal infinite components
might intersect.

First then, we describe precisely the respective end results of the ergodicity and trifurcation arguments.

\begin{lemma}\label{l:erg}
For each value of $d$, and for each $\lambda>0$, the number of infinite components in the graph $G=G'_{{\rm sec}}$ is
almost surely constant. (This constant might be $\infty$.)\qed
\end{lemma}

\begin{lemma}\label{l:tri}
Pick $r>0$ and $x\in R^d$. Let $T(x,r)$ be the event that the ball $B(x,r)$ is intersected by an infinite component $C$ of
$G=G'_{{\rm sec}}$ in such a way that, if all edges of $C$ intersecting $B(x,r)$ are removed, $C$ falls apart into a number of
components, of which at least three are infinite. Then $\Prb(T(x,r))=0$.\qed
\end{lemma}

Loosely speaking, if $\Prb(T(x,r))$ were strictly positive, then the expected number of occurrences of $T(x,r)$ in a large box
$A$ would be large, which in turn would mean that, with positive probability, the density of black points in $A$ would be at
least 2. The latter implication is purely combinatorial - see Lemma 3.2 of~\cite{MRbook}.

We next describe a useful coupling. Let $\cP_1$ and $\cP_2$ be two independent Poisson processes of intensity 1 in $\R^d$,
and let $\cP'_1$ and $\cP'_2$ be two independent Poisson processes of intensity $\lambda$, also in $\R^d$. Given $R>0$, construct
two more processes $\cP_3$ and $\cP'_3$ as follows. Outside $B(O,3R)$, let $\cP_3$ and $\cP'_3$ coincide with $\cP_1$ and $\cP'_1$
respectively. Inside $B(O,3R)$, for $\cP_3$, include each point of $\cP_1\cup\cP_2$ with probability $\tfrac12$, and for
$\cP'_3$, include each point of $\cP'_1\cup\cP'_2$ with probability $\tfrac12$. Then $\cP_3$ and $\cP'_3$ are both Poisson
processes in $\R^d$, of intensities 1 and $\lambda$, respectively. This coupling will be referred to, following~\cite{HaMe},
as the {\em special coupling}. It shows that, if an event $E$ occurs for an instance $(\cP_1,\cP_1')$, and if an event $F$
can be made to occur by removing some points of $(\cP_1,\cP_1')$ inside $B(O,3R)$, then $\Prb(E)>0\Rightarrow\Prb(F)>0$,
since the modified instance occurs with positive probability for $(\cP_1,\cP_2,\cP_1',\cP'_2)$. Its first application will
be in the proof of the following lemma.

\begin{lemma}\label{01infty}
For each value of $d$, and for each $\lambda>0$, the number of infinite components in the graph $G=G'_{{\rm sec}}$
is either almost surely 0, almost surely 1, or almost surely $\infty$.
\end{lemma}
\begin{proof}
By \Lm{erg}, we only have to show that, for each fixed $k\ge 2$, it is not the case that $G$ has, almost surely, exactly $k$ infinite
components. Suppose then, that, for some $k\ge 2$, $G$ has, almost surely, exactly $k$ infinite components. For some $r>0$, the
probability that each of these components $C_1,\ldots, C_k$ intersects $B(O,r)$ is strictly positive. Given some configuration
in which all $k$ infinite components $C_1,\ldots, C_k$ intersect $B(O,r)$, remove all the red points in $B(O,3r)$.  The effect
of this is that the $k$ components $C_1,\ldots,C_k$ merge to form a single infinite component. However, using the special coupling,
this shows that the probability of having a single infinite component is strictly positive, contradicting \Lm{erg}.
\end{proof}

\noindent We need one final technical lemma.

\begin{lemma}\label{l:final}
For sufficiently large $r$, the probability that there is any point of $\cP\setminus B(O,4r)$ that is closer to some point of
$B(O,3r)$ than to any point of $\cP'\setminus B(O,3r)$ is at most 0.1.
\end{lemma}
\begin{proof}
We can calculate the expected number of black vertices $v$ at distance at least $4r$ from $O$ whose nearest red point is at
distance more than $\|v\|-3r$ as
\begin{align*}
\int_{4r}^{\infty}e^{-\lambda \alpha_d(x-3r)^d}S_dx^{d-1}\,dx&=\int_{r}^{\infty}e^{-\lambda \alpha_dy^d}S_d(y+3r)^{d-1}\,dy\\
&\le \int_{r}^{\infty}e^{-\lambda \alpha_dy^d}S_d(4y)^{d-1}\,dy
\end{align*}
where $S_d=2\pi^{d/2}/\Gamma(d/2)$ and $\alpha_d=\pi^{d/2}/\Gamma(1+d/2)$ are the surface area and volume
respectively of a unit $d$ dimensional ball. The last integrand above is a polynomial times a (super-)
exponentially decreasing function, so the integral converges. Hence the integral can be made less than $0.1$
by suitable choice of $r$, and consequently so can the probability in the statement of the lemma.
\end{proof}

\noindent We are now ready for our final theorem.

\begin{theorem}\label{t:unique}
For each value of $d$, and for each $\lambda>0$, the number of infinite components in the graph $G=G'_{{\rm sec}}$ is either almost
surely 0, or almost surely 1.
\end{theorem}
\begin{proof}
In this proof we may assume that $\lambda>\lambda_c$, so that there is at least one infinite component, almost surely.

Suppose that, almost surely, $G$ has infinitely many infinite components. Then there exists an $r>0$ such that, with probability
at least 0.99, at least three infinite components intersect $B(O,r)$. \Lm{final} implies that $G\cap B(O,4r)^{\mathrm{c}}$ is
unaffected by the red points inside $B(O,3r)$. Now let $C_1,C_2$ and $C_3$ be three of the infinite
components intersecting $B(O,r)$. First, remove all black points not in these components from inside $B(O,4r)$. Second,
remove all the red points from $B(O,3r)$. The effect of this is that $C_1,C_2$ and $C_3$ merge into a single infinite
component $C$, while none of the other infinite components merge with $C$. But, in the new configuration, which has
positive probability of occurring (by the special coupling), $T(x,4r)$ occurs. This contradicts \Lm{tri}.
\end{proof}

\section{Concluding Remarks}

We have presented several methods to calculate bounds on five percolation thresholds in the Poisson secrecy graph.
While the rigorous bounds are still rather loose, the high-confidence lower bounds derived here are much tighter.

\section*{Acknowledgments}

The work of the second author was in part supported by the U.S.~NSF (grants CCF 728763, CNS 1016742)
and the DARPA/IPTO IT-MANET program (grant W911NF-07-1-0028).

\end{document}